\documentclass{article}

\usepackage[colorlinks=true,breaklinks]{hyperref}
\usepackage{tipa, amsmath, amssymb, amsthm, graphicx, psfrag}
\usepackage[theoremfont]{newtxtext}
\usepackage[smallerops]{newtxmath}

\newcommand{\mm}{\text{\textltailm}}
\newcommand{\ks}{\text{\texthtk}}
\newcommand{\ki}{\kappa}
\newcommand{\abs}[1]{\lvert #1 \rvert}
\newcommand{\norm}[1]{\lVert #1 \rVert}
\newcommand{\V}[1]{\boldsymbol{#1}}
\newcommand{\grad}{\mathop{\nabla}}
\newcommand{\laplace}{\mathop{\Delta}}
\renewcommand{\div}{\mathop{\mathrm{div}}}
\DeclareMathOperator*{\esssup}{ess\,sup}
\DeclareMathOperator*{\supp}{supp}

\DeclareMathOperator*{\diam}{diam}

\newcommand{\eg}{e.g.}

\newtheorem{thm}{Theorem}
\newtheorem{lem}[thm]{Lemma}
\newtheorem{cor}[thm]{Corollary}
\theoremstyle{definition}
\newtheorem{rmk}{Remark}
\newtheorem{dfn}{Definition}

\title{On periodic motions of a harmonic oscillator interacting with
  incompressible fluids}

\author{Giusy Mazzone and Mahdi Mohebbi}
\date{\small{\em Department of Mathematics and Statistics 
\\
Queen’s University
\\
Kingston, ON K7L 3N6}}

\begin{document}

\maketitle
\begin{abstract}
  We consider a mass-spring system immersed in an incompressible fluid
  flow governed by the Navier-Stokes equations subject to a prescribed
  time-periodic flow rate (and possibly external time-periodic body forces on
  the fluid and the mass). We show that, with no restriction on the period of
  the flow rate (and of the external forces), when the flow rate is ``small'',
  there exits a weak time-periodic solution to the coupled system. Under some
  more regularity and ``smallness'' conditions on the flow rate (and the
  external forces) we also show that these solutions are, indeed, strong
  solutions.
\end{abstract}

\section{Introduction}

We consider the interaction between a harmonic oscillator (consisting of a mass
and a spring) and a fluid occupying an infinite channel. The fluid inside
the channel is driven by a prescribed periodic flow rate (with period $T$). We
investigate whether the coupled fluid-oscillator system admits a periodic
motion (with the same period $T$). We are not imposing any restriction on the
period of the flow rate and, in particular, this period can be the natural
frequency of the oscillator (that is the frequency at which the mass-spring
system will oscillate due to initial perturbations and in the absence of
external forces). Physical intuition suggests that, under a prescribed
time-periodic flow rate, the fluid would exert on the oscillator a time-periodic
force having frequency matching the natural frequency of the oscillator. In
this scenario, the phenomenon of resonance would occur (since the oscillator is
undamped), and the generic motion of the oscillator would be characterized by
oscillations with increasing amplitude.  In mathematical terms, this means that
no periodic motion would exist. In this paper, we show that this intuition is
not correct and, in fact, the fluid dissipation provides sufficient damping to
guarantee the existence of such periodic motions for the fluid-oscillator
system, no matter what the period of the flow rate is. From a physical point of
view such a system can be an abstraction of many engineering structures, energy harvesting devices and
\emph{in situ} medical devices \cite{MR4419355, bloodvalve1, bloodvalve2, wave1}.

The equations governing the motion of the fluid-oscillator system are given by
the coupling of the Navier-Stokes equations for the fluid and the balance of
linear momentum for the mass-spring system, \eqref{eq:noninertial}. In Theorem
\ref{thm:weak} we show that, for an arbitrary period $T$, under quite general
external $T$-periodic body forces on the mass and the fluid and when the
prescribed $T$-periodic flow rate is ``small'' there is at least one weak
solution to the coupled system. These solutions tend to the generalized
$T$-periodic Poiseuille flow \cite{MR2196495} at channel inlets/outlets
(Remark \ref{rmk:assymptotics}). In Theorem \ref{thm:strong}, when the flow
rate and forces are more regular and under some additional ``smallness''
conditions, we show that there are strong solutions satisfying the equations
almost everywhere.

From a mathematical point of view, the main difficulty in finding time-periodic
solutions to this system is that the governing equations are only ``partially''
dissipative, that is, the ``natural'' energy inequality of the system (obtained
from the balance of kinetic energy, see \eqref{eq:E}) lacks a dissipative term
corresponding to the potential energy of the spring, $\ks \abs{z}^2$. As such,
standard well-known techniques to show the existence of periodic solutions for
nonlinear PDEs (\eg~\cite{MR0120968}) cannot be applied. These techniques, as
an important part of their argument, consider the initial value problem and
show that it is possible to choose the initial values in a bounded set, $A$,
such that the Poincar\'e map, taking these initial values to the corresponding
solutions at time $T$, is compact with the target set $A$. The Poincar\'e map
then admits fixed points which, in turn, yield to periodic solutions to the
system. In the absence of a ``complete'' energy inequality, where the
dissipation term is proportional to the energy itself, constructing such a set
$A$, is not possible. Nevertheless, we show that it is possible to complete the
dissipation term by considering a ``particular'' energy inequality,
\eqref{eq:G}; Yet, our proof takes an unconventional path, in that we use the
Leray-Schauder fixed point argument in finite dimensional unbounded sets.

More precisely, our strategy consists of using the Galerkin method along with
suitable energy estimates to show the existence of weak periodic solutions but
not through the fixed points of the Poincar\'e map discussed above. Instead,
the basic idea is to consider the linearized problem (at each Galerkin level),
where essentially the nonlinear term in the Navier-Stokes equation, $\V v \cdot
\V\grad v$ in \eqref{eq:noninertial0phi}, is replaced with $\V{\tilde{v}} \cdot
\V\grad v$, for some given function $\V{\tilde{v}}$. The existence of periodic
solution to this linearized problem follows easily from theorems available in
the context of ordinary differential equations. Next, we consider the map
$\Phi$ (see \eqref{eq:Phi}), that maps any $T$-periodic $\V{\tilde{v}}$ to the
$T$-solution of the linearized problem. The existence of periodic solutions to
the original nonlinear problem is then established once we show that $\Phi$ has
a fixed point. The ``particular'' energy inequality, \eqref{eq:G}, can be used
to show that the set of fixed points (and their straight line homotopy) is
bounded, which (along some other properties for $\Phi$) guarantees the
existence of a fixed point by Leray-Schauder principle. This procedure requires
obtaining, explicitly, some specific energy estimates (see \eg
\eqref{eq:energy}), that are not needed if a Poincar\'e map argument is used
and play a fundamental role in showing the higher regularity of the solutions.
Both the derivation of the ``particular'' energy inequality and the
unconventional proof, just outlined, through a Leray-Schauder fixed point
argument follow ideas previously developed for other problems concerning the
existence of periodic solutions to partially dissipative systems in
magnetoelasticity \cite{MR3092957}.

For the linear case when the fluid is governed by the Stokes equations, a
related problem have been considered in \cite{MR4559728}. However, one shall
note that when the problem is linear and the solutions to the initial value
problem are unique, the existence of periodic solutions and its relation to
the occurrence of resonance can be addressed satisfactorily (see
\eg~\cite{MR3329019}). However, in nonlinear problems (and specifically for the
problem considered here) the existence of periodic solutions (even strong
solutions) is not known to be a sufficient (and a necessary) condition for the
phenomenon of resonance not to occur. This is because the system under a
particular external force may have a periodic solution which can be viewed as a
solution to the corresponding initial boundary value problem with a specific
initial condition and, at the same time, there are other initial conditions
for which the initial value problem will have unbounded solutions (say, in the
energy norm). With this consideration, to remove the possibility of the
occurrence of resonance, one needs a type of energy inequality for all the
solutions in a certain regularity class corresponding to periodic external
forces (of a given regularity class); And this, will just show that
resonance will not occur in these assumed regularity classes.

A similar problem to what is considered here, has been investigated in
\cite{bonheureGaldi}, where a system of mass-spring is considered in
interaction with an incompressible fluid filling the whole domain
$\mathbb{R}^3$. The fluid is subject to a prescribed uniform time-periodic
velocity at infinity. It is shown that, also for this case, weak time-periodic
solutions exist with no restriction on the period. However, the results are
obtained in the absence of direct external forces on the mass and/or the
fluid. Our method to show the existence of strong solutions may be applied also
to the whole domain case to obtain strong solutions under some restrictions on
the prescribed uniform velocity at infinity.

\section{Formulation}

Consider an incompressible Newtonian fluid in an infinite channel interacting
with a harmonic oscillator, as shown in Figure.~\ref{fig:infchan}. The harmonic
oscillator is composed of a spring with stiffness constant, $\ks$,
attached to a rigid body of mass, $\mm$. The force exerted by the spring on the
rigid body (referred to as the ``mass'' in what follows) is modeled by Hooke's
law. Without loss of generality, we assume the mass is constrained to move
horizontally, and ignore the effects of gravitational forces (see Remark
\ref{rmk:forces}). Consider an inertial Cartesian coordinate system $\{O', \V
e'_i\}$, $i=1, 2$ and $3$, with the origin $O'$ coinciding with the end of the spring at its equilibrium. Further, assume
the channel, $\mathcal{C}$, is a straight channel along $\V e'_1$ with constant
cross-section, $\Pi \subset \mathbb{R}^2$; precisely, $\mathcal{C} = \Pi \times
\mathbb{R}$. Let $\mathcal{B}(t) \in \mathbb{R}^3$ denote the region occupied
by the rigid body at time $t$ and let $\Gamma(t) = \partial\mathcal{B}$. Then
the volume occupied by the fluid at time $t$ is $\Omega(t) = \mathcal{C}
\setminus \overline{\mathcal{B}(t)}$. Denote by $y_i$, the $i$-th coordinate of
a point $\V y \in \Omega(t)$ and by $\V u(\V y, t)$ and $z(t)$ the velocity of
the fluid and the displacement of the mass (from spring's equilibrium),
respectively. Assume that the fluid is subject to move under a prescribed
(time-)periodic flow rate, $\phi(t)$, with period $T>0$. The governing
equations for the coupled system of the fluid and the harmonic oscillator are
given by
\begin{equation} \label{eq:master}
  \begin{aligned}
  &\left. \begin{aligned}
    &\frac{\partial \V u}{\partial t} + \V u \cdot \grad\V u
    = \frac1\rho \div\V T(\V u, p), \\
    &\div \V u = 0,
  \end{aligned} \right\} \quad \text{in } \Omega(t) \times \mathbb{R}, \\
  & \int_{\mathcal{S}} \V u(t) \cdot \V n_{s} \;dS = \phi(t), \qquad \forall
    t \in \mathbb{R}, \\
  &\;\;\mm\frac{d^2z}{dt^2} + \ks z = \int_{\Gamma(t)} \V e'_1 \cdot \V T(\V u,
    p) \cdot \V n \;dS.
  \end{aligned}
\end{equation}
In the above equations, $\rho$ is the constant density of the fluid and $\V
n=\V n(\V y, t)$ denotes the unit outward normal vector, to the boundary
$\Gamma(t)$ of the body. $\V T$ indicates the Cauchy stress tensor for an
incompressible Newtonian fluid:
\begin{equation*}
  \V T(\V u, p) = -p\V I + 2\mu \V D(\V u), \qquad \V D(\V u) = \frac12
  (\grad\V u + (\grad\V u)^T),
\end{equation*}
where $\mu$ is the (constant) dynamic viscosity coefficient of the fluid and
$p=p(\V y,t)$ is the pressure field. $\mathcal{S} \subset \overline{\Omega'}$,
for some bounded $\Omega' \subset \Omega(t)$, is any orientable surface with
normal $\V n_s$ such that $\partial\mathcal{S} \subset \Sigma$, where $\Sigma =
\partial\Pi \times \mathbb{R}$ denotes the boundary of the channel and is
independent of time.
\begin{figure}[b] 
  \small
  \psfrag{!}{$\V e_1$} \psfrag{"}{$O$} \psfrag{#}[bc][bc]{$\V e'_1$}
  \psfrag{\$}{$O'$} \psfrag{\%}{$\ks$} \psfrag{&}{$\mm$}
  \psfrag{'}{$\Gamma(t)$} \psfrag{\(}{$\Sigma$} \psfrag{\)}{} 
  \psfrag{*}{$\V n$} \psfrag{+}{$\Omega(t)$}
  \centering
  \includegraphics[width=0.8\textwidth]{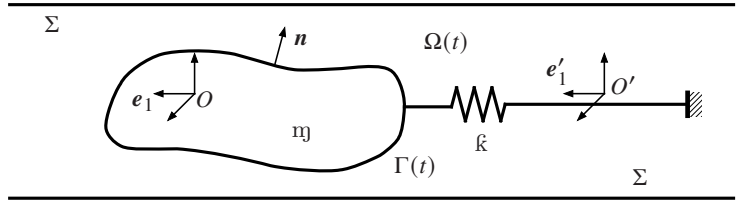}
  \caption{Infinite channel configuration.}
  \label{fig:infchan}
\end{figure}
Assuming no-slip conditions on the fluid boundaries, we are
concerned with the existence of $T$-periodic solutions to \eqref{eq:master},
for any period $T$ and ``small'' flow rates $\phi(t)$ (see
\eqref{eq:smallnessWeak}). We append the following boundary and
periodicity conditions for all $t \in \mathbb{R}$:
\begin{equation} \label{eq:bc}
\begin{aligned}
  &\V u(t) = \frac{dz}{dt}\V e'_1, \qquad\text{on } \Gamma(t), \\
  &\V u(t) = 0,  \qquad\quad\text{on } \Sigma, \\
  &\Omega(t+T) = \Omega(t), \qquad \V u(\V y, t+T) = \V u(\V y, t) \quad
  \text{and} \quad z(t+T) = z(t).
\end{aligned}
\end{equation}
To remove the inconvenience of the unknown time dependent domains in the above
formulation, we consider a new frame, $\mathcal{N}$, with Cartesian coordinate
system $\{O, \V e_i\}$, attached to the mass $\mm$. Assume, without loss of
generality, that $O$ is at an interior point $\mathcal{B}(t)$ and that
$\mathcal{N}$ is oriented in such a way that $\V e_i$ is parallel to $\V e'_i$
for all $i$ and let $\V x$ denote the position vector of a point in the new
non-inertial frame. It can be shown that the change of variables $\V y
\rightarrow \V x$ defined by
\begin{equation*}
  \V y = \V x + z(t)\V e_1 = \V x + z(t) \V e'_1,
\end{equation*}
transforms \eqref{eq:master} and \eqref{eq:bc} into the following boundary
value problem (where all the functions involving the space variables are
understood to be functions of the $\V x$ variable):
\begin{equation} \label{eq:noninertial}
  \begin{aligned}
  &\left. \begin{aligned}
    &\frac{\partial \V u}{\partial t} + (\V u - \frac{dz}{dt}\V e_1) \cdot
      \grad\V u = \frac1\rho \div\V T(\V u, p), \\
    &\div \V u = 0,
  \end{aligned} \right\} \quad \text{in } \Omega \times \mathbb{R}, \\
    &\;\;\mm\frac{d^2z}{dt^2} + \ks z = \int_\Gamma \V e_1 \cdot \V T(\V u, p)
    \cdot \V n \;dS, \\
    & \int_{\mathcal{S}} \V u(t) \cdot \V n_{s} \;dS= \phi(t), \qquad \forall
    t \in \mathbb{R}, \\
    & \V u = \frac{dz}{dt} \V e_1, \quad \text{ on } \Gamma, \\
    & \V u = 0, \quad \text{ on } \Sigma, \\
    & \V u(t+T) = \V u(t), \quad z(t+T) = z(t).
  \end{aligned}
\end{equation}
In the above, $\Omega$ and $\Gamma$ denote the time-independent domain of the
fluid and the boundary of the rigid body, respectively, referenced in
$\mathcal{N}$ (and we also have that the region occupied by the rigid body,
$\mathcal{B}$, is time-independent in the frame $\mathcal{N}$). Concerning the
regularity of the boundary, we assume that $\Omega$ is a Lipschitz domain when
we are concerned with the weak solutions to the above system and that $\Omega$
is a domain of class $C^2$ when considering the strong solutions to
\eqref{eq:noninertial}.
\begin{rmk}
If the channel does not have a constant cross-section, the above transformation
(or any other) will not render the domain time-independent. This may be handled
by a more involved mathematical analysis, however, with no significant gained
advantage from a physical point of view. As long as the methods presented here
are concerned, a crucial estimate depends on the property \ref{itm:VFar}
below, of the ``flux carrier'' and its particular form \eqref{eq:chiForm},
which hold only if the cross-section of the ``exits'' are constant (although,
not necessarily the same) whenever $\abs{\V x} > C$, for some $C>0$.
\end{rmk} \label{sec:fluxCarrier}

\section{Function Spaces and Preliminaries}
For $\Omega$, $\Sigma$ and $\Gamma$ as in the previous section and $\Omega'
\subseteq \Omega$, we denote by $L^p(\Omega')$ and $W^{m,p}(\Omega')$ the usual
Lebesgue and Sobolev spaces with norms $\norm{\cdot}_{L^p(\Omega')}$ and
$\norm{\cdot}_{W^{m,p}(\Omega')}$, respectively. In $L^2(\Omega')$, when there
is no confusion, we use the relaxed notation $\norm{\cdot}$ for the norm and
$(\cdot,\cdot)$ for the inner product. $W^{m,p}_{loc}(\overline{\Omega})$
($L^p_{loc}(\overline{\Omega})$) denotes the space of functions $u$ such that
$u \in W^{m,p}(\Omega')$ ($u \in L^p(\Omega')$) for all bounded $\Omega'
\subset \Omega$. Vector and tensor fields are denoted by boldface letters and,
with an abuse, we employ the same notation for the spaces of scalar and vector
functions.

Let 
\begin{equation*}
  \mathcal{D}_0^\infty = \{\V\psi \in C_0^\infty(\Omega\cup\Gamma):
  \,\div\V\psi = 0 \text{ in } \Omega, \;\;\V\psi = \beta\V e_1 \text{ on }
  \Gamma, \text{ for some } \beta \in \mathbb{R} \}.
\end{equation*}
Where $C_0^\infty(\Omega\cup\Gamma)$ indicates the space of smooth functions
with compact support in $\Omega\cup\Gamma$. We denote by $\mathcal{D}$ and
$\mathcal{D}^1$ the Banach spaces obtained as the completion of
$\mathcal{D}_0^\infty$ with respect to the norms of $L^2(\Omega)$ and
$W^{1,2}(\Omega)$, respectively.
\begin{rmk} \label{rmk:DvGradv}
  For any $\V\psi \in \mathcal{D}_0^\infty$, it can be shown that
  $\norm{\grad\V\psi} = \sqrt{2}\norm{\V D(\V\psi)}$ and so by Poincar\'e
  inequality, $\norm{\V D(\V\psi)}$, $\norm{\grad\V\psi}$ and
  $\norm{\V\psi}_{W^{1,2}(\Omega)}$ are all equivalent norms. This can be
  extended by a density argument to all functions in $\mathcal{D}^1$.
\end{rmk}
The following Lemma provides an important estimate for the type of nonlinear
terms that will be encountered later:
\begin{lem}\label{lem:nonlinearBound}
  There is a constant $c_q = c_q(\Omega, \mu, \rho)$, such that for all $\V\psi
  \in \mathcal{D}_0^\infty$ with $\V\psi|_\Gamma = \beta\V e_1$,
  \begin{equation*}
    \abs{((\V\psi-\beta\V e_1) \cdot \grad\V V, \V\psi)} \le
    c_q \norm{\phi}_{W^{1,2}_T} \norm{\grad\V\psi}^2.
  \end{equation*}
\end{lem}
\begin{proof}
  Following \cite[pp.~321]{MR2196495} and using H\"older's inequality
  \begin{align*}
    \abs{((\V\psi-\beta\V e_1) \cdot \grad\V V, \V\psi)} &\le
    \int_{\Omega_0} \abs{(\V\psi-\beta\V e_1) \cdot \grad\V V \cdot
      \V\psi} \;dx + \int_{\Omega\setminus\Omega_0}
    \abs{\V\psi \cdot \grad\V V \cdot \V\psi} \;dx \\
    &\le \norm{\V\psi-\beta\V e_1}_{L^4(\Omega_0)} \norm{\grad\V
      V}_{L^2(\Omega_0)} \norm{\V\psi}_{L^4(\Omega_0)} \\
    &\hphantom{\le} + \int_{-\infty}^{-X_0} \int_\Pi \abs{\V\psi \cdot \grad\V
      \chi \cdot \V\psi} \, dS\,dx_1 + \int_{X_0}^{+\infty} \int_\Pi
    \abs{\V\psi \cdot \grad\V \chi \cdot \V\psi} \, dS\,dx_1 \\
    &\le \norm{\V\psi-\beta\V e_1}_{L^4(\Omega_0)} \norm{\grad\V
      V}_{L^2(\Omega_0)} \norm{\V\psi}_{L^4(\Omega_0)} \\
    &\hphantom{\le} + \norm{\grad\V \chi}_{L^2(\Pi)}
    \norm{\V\psi}^2_{L^4(\Omega\setminus\Omega_0)}.
  \end{align*}
  Noting that $\V\psi - \beta\V e_1$ vanishes on $\Gamma$, and $\V\psi$
  vanishes on $\Sigma$, the statement follows from Sobolev embedding
  theorem, the Poincar\'e inequality, \eqref{eq:carrierP}$_1$ and
  \eqref{eq:carrierPext}$_1$.
\end{proof}
\begin{cor} \label{cor:nonlinearBound}
  Let $\V u \in W^{1,2}(\Omega')$ for some $\Omega' \subseteq \Omega$
  such that $\V u|_\Sigma = 0$ and $\V u|_\Gamma = \beta \V e_1$, then
  \begin{equation*}
    \left|\int_{\Omega'}(\V u-\beta\V e_1) \cdot \grad\V V \cdot \V
    u \,dx \right| \le c_q \norm{\phi}_{W^{1,2}_T}
    \norm{\grad\V u}^2_{L^2(\Omega')}.
  \end{equation*}
\end{cor}
The space of smooth periodic functions in $\mathbb{R}$ with period $T>0$, is
denoted by $C_T^\infty(\mathbb{R})$ and the completion in $W^{m,p}([0,T])$
(respectively, $L^p([0,T])$) of the restriction of such functions to $[0,T]$,
will be indicated with $W_T^{m,p}$ (respectively, $L_T^p$). Given the Banach
space $X$ with norm $\norm{\cdot}_X$, a function $f: [0,T] \rightarrow X$
belongs to $L^q(0,T;X)$ if,
\begin{equation*}
  \begin{cases} \displaystyle
    \left(\int_0^T \norm{f}_X^q \;dt\right)^{1/q} < \infty, & 1 \le q < \infty,
    \\ \displaystyle
    \esssup_{t\in[0,T]} \norm{f}_X < \infty, & q = \infty.
  \end{cases}
\end{equation*}

It is, of course, understood that by a $T$-periodic function $f \in L^q(0,T;X)$
(or $f \in L_T^p$), we mean $f: \mathbb{R} \rightarrow X$ such that
$\norm{f}_{L^q(s,t;X)} = \norm{f}_{L^q(T+s,T+t;X)}$ (respectively,
$\norm{f}_{L^p([s,t])} = \norm{f}_{L^p([T+s,T+t])}$) for all $t<s$.

\section{Re-formulation and Physical Considerations} \label{sec:flux}

For $X>0$, let $\Omega^X = \{\V x \in \Omega: x_1>X\}$ and $\Omega^{-X} = \{\V
x \in \Omega: x_1 < -X\}$. Also, let $X_0 = \diam(\mathcal{B})+1$, where
$\diam(\mathcal{B})$ is the diameter of the domain occupied by the mass
$\mm$. Let $\Omega_0$ = $\Omega \setminus (\overline{\Omega^{X_0}
\cup \Omega^{-X_0}})$. Then following \cite[pp.~316--317]{MR2196495} we
consider a ``flux carrier'' $\V V$ satisfying the following conditions:
{\renewcommand{\theenumi}{(\roman{enumi})}
 \renewcommand{\labelenumi}{\theenumi}
  \begin{enumerate}
  \item $\V V \in W^{1,2}(0,T;L^2_{loc}(\Omega)) 
    \cap L^2(0,T;W^{2,2}_{loc}(\Omega))$.
  \item $\V V(t+T)=\V V(t)$ for all $t\in \mathbb R$.
  \item $\div\V V = 0$ in $\Omega$.
  \item $\V V = 0$ on $\Gamma \cup \Sigma$.
  \item $\V V(\V x,t) = \V \chi(\V x,t)$ for all $t\in \mathbb R$ and $|x_1|\ge
    X_0$. \label{itm:VFar}
  \end{enumerate}
}
The vector field $\V\chi$ is the \emph{generalized $T$-periodic Poiseuille
flow} and is a $T$-periodic solution to the Navier-Stokes equations (with
no-slip boundary conditions) in the infinite cylindrical channel:
\begin{equation}
  \begin{aligned}
    &\left.\begin{aligned}
      &\frac{\partial\V\chi}{\partial t} + \V\chi \cdot \grad\V\chi -
      \frac{\mu}{\rho}\laplace\V\chi = \grad q, \\
      &\div\V\chi = 0,
    \end{aligned}\right\} \text{ in }\Pi \times \mathbb{R},\\
    &\V\chi(\V x, t) = 0, \qquad \V x \in \Sigma, \;t \in \mathbb{R}\\
    &\V\chi(\V x, t) = \V\chi(\V x, t+T), \qquad (\V x, t) \in \Pi
    \times \mathbb{R},
  \end{aligned}
\end{equation}
satisfying the following two properties
\begin{gather}
    \V \chi(\V x,t)=\chi(x_2,x_3,t)\V e_1, \label{eq:chiForm}
    \\
    \int_{\mathcal S} \V \chi(t)\cdot \V n_s \,dS =\phi(t), \qquad
    \forall t \in \mathbb{R}, \nonumber
\end{gather}
for a given $T$-periodic flow rate $\phi$. The existence (and uniqueness) of
$\V \chi$ corresponding to $\phi\in W^{1,2}_T$ and its higher regularity when
$\phi \in W^{3,2}_T$, has been established in \cite[Theorem 1 and
  Remark1]{MR2196495}, and, in fact, we have the following estimates:
\begin{equation}\label{eq:carrierP}
  \begin{aligned}
    \norm{\V \chi}_{L^2(0,T;W^{2,2}(\Pi)},\ 
    \norm{\V \chi}_{C((0,T);W^{1,2}(\Pi))},\ 
    \norm{\V \chi}_{W^{1,2}(0,T;L^{2}(\Pi))}
    &\le c_v\norm{\phi}_{W^{1,2}_T}, \\
    \norm{\V \chi}_{W^{1,2}(0,T;W^{2,2}(\Pi))},\ 
    \norm{\V \chi}_{C^1((0,T);W^{1,2}(\Pi))},\ 
    \norm{\V \chi}_{W^{2,2}(0,T;L^{2}(\Pi))}
    &\le c'_v\norm{\phi}_{W^{2,2}_T}, \\
    \norm{\V \chi}_{W^{2,2}(0,T;W^{2,2}(\Pi))},\ 
    \norm{\V \chi}_{C^2((0,T);W^{1,2}(\Pi))},\ 
    \norm{\V \chi}_{W^{3,2}(0,T;L^{2}(\Pi))}
    &\le c''_v\norm{\phi}_{W^{3,2}_T},
  \end{aligned}
\end{equation}
and the flux carrier, $\V V$, satisfies \cite[pp.~316--317]{MR2196495}
\begin{equation}\label{eq:carrierPext}
  \begin{aligned}
    \norm{\V V}_{L^2(0,T;W^{2,2}(\Omega_0))},\ 
    \norm{\V V}_{C((0,T);W^{1,2}(\Omega_0))},\ 
    \norm{\V V}_{W^{1,2}(0,T;L^{2}(\Omega_0))}
    &\le c_v\norm{\phi}_{W^{1,2}_T}, \\
    \norm{\V V}_{W^{1,2}(0,T;W^{2,2}(\Omega_0))},\ 
    \norm{\V V}_{C^1((0,T);W^{1,2}(\Omega_0))},\ 
    \norm{\V V}_{W^{2,2}(0,T;L^{2}(\Omega_0))}
    &\le c'_v\norm{\phi}_{W^{2,2}_T}, \\
    \norm{\V V}_{W^{2,2}(0,T;W^{2,2}(\Omega_0))},\ 
    \norm{\V V}_{C^2((0,T);W^{1,2}(\Omega_0))},\ 
    \norm{\V V}_{W^{3,2}(0,T;L^{2}(\Omega_0))}
    &\le c''_v\norm{\phi}_{W^{3,2}_T},
  \end{aligned}
\end{equation}
where $c_v$, $c'_v$ and $c''_v$ are positive constants depending at most on
$\rho$, $\mu$, and $\Omega$.

Then $(\V u, p, z)$ is a solution of \eqref{eq:noninertial} if and only if ($\V
v = \V u - \V V, p, z)$ satisfies
\begin{equation} \label{eq:noninertial0phi}
  \begin{aligned} 
    &\left. \begin{aligned}
      &\begin{aligned}
         \frac{\partial\V v}{\partial t} + (\V v- \frac{dz}{dt}\V e_1) \cdot
         \grad\V v = &\frac1\rho \div\V T(\V v, p) \\
         &-\V V \cdot \grad\V v -(\V v - \frac{dz}{dt}\V e_1) \cdot \grad\V V +
         \V f,
       \end{aligned}\\
      &\div\V v = 0,
    \end{aligned} \right\} \text{ in } \Omega \times \mathbb{R}, \\
    &\;\;\mm\frac{d^2z}{dt^2} + \ks z = 
    \int_{\Gamma} \V e_1 \cdot \V T(\V v, p) \cdot \V n \;dS + g, \\
    & \int_{\mathcal{S}} \V v(t) \cdot \V n_s \;dS= 0, \quad \forall t \in
    \mathbb{R}, \\
    &\;\;\V v = \frac{dz}{dt}\V e_1, \qquad \text{ on } \Gamma,\\
    &\;\;\V v = 0, \qquad\quad\text{ on } \Sigma, \\
    &\;\;\V v(t+T) = \V v(t), \quad z(t+T)=z(t).
  \end{aligned}
\end{equation} 
In the above,
\begin{equation} \label{eq:forces}
  \begin{aligned}
    \V f(\V x, t) &= \frac{\mu}{\rho} \laplace\V V - \V V \cdot \grad\V V 
    -\frac{\partial\V V}{\partial t},\\
    g(t) &= \mu \int_{\Gamma} \V e_1 \cdot (\grad\V V + (\grad\V V)^T)
    \cdot \V n \;dS.
  \end{aligned}
\end{equation}
\begin{rmk} \label{rmk:phiTrivial}
  In the trivial case $\phi(t) \equiv \V 0$, by uniqueness \cite[Theorem
    1]{MR2196495}, $\V\chi \equiv \V 0$. In this case, we choose the extension
  $\V V \equiv \V 0$ although there may be nonzero corresponding extensions. 
\end{rmk}
\begin{rmk} \label{rmk:fvHomo}
  It should be noted that in \eqref{eq:forces}$_1$, $\V f \equiv \V 0$ or (more
  generally) $\V f = \grad q \in L^2_{loc}(0,\infty;L^2_{loc}(\Omega))$ implies
  that $\V\chi \equiv \V 0$ and hence by the remark above $\V V = \V 0$ (and
  $\phi(t) = 0$). To see this, consider $X_1$ and $X_2$ such that $X_2 > X_1 >
  X_0$, then on $\Omega'=\Omega^{X_2} \setminus \overline{\Omega^{X_1}}$, $\V f
  = \V 0$ yields
  \begin{equation*}
    \frac\mu\rho \laplace \V\chi - \frac{\partial \V\chi}{\partial t} = 0,
  \end{equation*}
  taking the ($L^2$-)inner product of the above with $\V\chi$ in $\Omega'$ and
  using Poincar\'e inequality, we find
  \begin{equation*}
    \frac{d\norm{\V \chi}_{L^2(\Omega')}^2}{d t} - c \norm{\V
    \chi}_{L^2(\Omega')}^2 = 0.
  \end{equation*}
  But the only $T$-periodic solution to the above equation is $\norm{\V
    \chi}_{L^2(\Omega')} = 0$.
\end{rmk}
\begin{rmk} \label{rmk:forces}
  $\V f$ and $g$ do not need to be of the form in \eqref{eq:forces}. In fact,
  $g$ can be modified to include an external $T$-periodic force, $\tilde{g}$,
  on the mass $\mm$; and as long as mathematical analysis is concerned, $\V f$
  can also be modified to include any suitable $T$-periodic body force,
  $\V{\tilde{f}}$, acting on the fluid:
  \begin{equation} \tag{\ref{eq:forces}$'$} \label{eq:altForces}
  \begin{aligned}
    \V f(\V x, t) &= \frac{\mu}{\rho} \laplace\V V - \V V \cdot \grad\V V 
    -\frac{\partial\V V}{\partial t} + \V{\tilde{f}},\\
    g(t) &= \mu \int_{\Gamma} \V e_1 \cdot (\grad\V V + (\grad\V V)^T)
    \cdot \V n \;dS + \tilde{g}.
  \end{aligned}
  \end{equation}
  However, if $\V{\tilde{f}} \ne 0$ is originally present in
  \eqref{eq:master}$_1$, then, without loss of generality, we shall choose $\V
  V$ such that,
  \begin{equation*}
    \frac{\partial\V V}{\partial t} + \V V \cdot \grad\V V - \frac{\mu}{\rho}
    \laplace\V V \ne \grad q + \V{\tilde{f}}, 
  \end{equation*}
  for any $\grad q \in L^2_{loc}((0,\infty)\times\Omega)$. This ensures, by
  Remark~\ref{rmk:fvHomo}, that the corresponding ``homogeneous'' system to
  \eqref{eq:master} (or \eqref{eq:noninertial0phi}) will be obtained only when
  all the external forcing mechanisms are identically zero: $\V{\tilde{f}}
  \equiv 0$, $\V V \equiv 0$ ($\phi(t) \equiv 0$) and $\tilde{g} \equiv
  0$.
  
  It should be emphasized that even in the presence of $\V{\tilde{f}}$ and
  $\tilde{g}$, \emph{the flow rate, $\phi(t)$, is still prescribed}. In this
  case, from a physical point of view, there are several driving mechanisms
  present. Concerning the regularity of the external forcing, in the case of
  weak solutions, we assume that,
  \begin{align}\label{eq:forcesRegularityWeak}
    \phi \in W^{1,2}_T, &&\V{\tilde{f}} \in L^2(0,T;L^2(\Omega)) &&
    \tilde{g} \in L^2([0,T]),
  \end{align}
  whereas for strong solutions, we assume
  \begin{align} \label{eq:forcesRegularityStrong}
    \phi \in W^{3,2}_T, &&\V{\tilde{f}} \in W^{1,\infty}(0,T;L^2(\Omega)), &&
    \tilde{g} \in W^{1,\infty}_T.
  \end{align}
\end{rmk}
\begin{rmk}\label{rmk:suppf}
  Considering $\phi \in W^{1,2}_T$ (for the case of weak solutions discussed
  below), from property \ref{itm:VFar} above, it follows that $\V V(\V x,t) =
  \V\chi(\V x,t)$ for all $\V x\in \Omega\setminus\Omega_0$ and $t\in \mathbb
  R$. We recall that $\V \chi$ solves the time-periodic Navier-Stokes equations
  in $\Omega\setminus\Omega_0$ with the corresponding pressure field
  \cite[Section 2]{MR2196495},
  \begin{equation}\label{eq:pressureP}
    \begin{split}
      \tilde{p}(t) &= -\psi(t)x_1,
      \\
      \psi(t) &:= \frac{1}{\abs{\Pi}} \left(\frac{d\phi(t)}{dt} -
      \frac{\mu}{\rho}\int_\Pi\laplace\chi \;dx \right). 
    \end{split}
  \end{equation}
  So, in the absence of the external forces $\V{\tilde{f}}$ and $\tilde g$,
  redefining the forcing terms in \eqref{eq:forces} as:
  \begin{equation}\label{eq:forcePext}
    \begin{aligned}
      \V f(\V x,t) &= \frac{\mu}{\rho} \laplace\V V - \V V \cdot \grad\V V -
      \frac{\partial\V V}{\partial t} - \nabla \tilde p, \\
      g(t) &= \mu \int_\Gamma \V e_1 \cdot (\grad\V V + (\grad\V V)^T) \cdot \V
      n \,dS - \rho \int_\Gamma \tilde p n_1\,dS,
    \end{aligned}
  \end{equation}
  and adding the pressure term $\rho\tilde p \V I$ in the Cauchy stress tensors
  in \eqref{eq:noninertial0phi}, we get that
  \begin{equation} \label{eq:fSupp}
    \supp \V f\subset \Omega_0.
  \end{equation}
  In addition, \eqref{eq:carrierPext}$_1$ and usual estimates on the nonlinear
  term lead to
  \begin{equation*}
    \norm{\V f}_{L^2(0,T;L^2(\Omega))}\le c_f\norm{\phi}_{W^{1,2}_T},
  \end{equation*}
  for some positive constant $c_f=c_f(\rho,\mu,\Omega)$. Also, using the trace
  inequality and \eqref{eq:pressureP}, we have the following bound for the
  force $g$ in \eqref{eq:forcePext}$_2$:
  \begin{align*}
    \abs{g}^2 &= \left| \mu \int_\Gamma \V e_1 \cdot
    (\grad\V V+(\grad\V V)^T) \cdot \V n \; dS - \rho
    \int_\Gamma \tilde p n_1 \;dS \right|^2 \\
    &\le a \left( \int_\Gamma \abs{\grad\V V}^2 \;dS
    + \left| \frac{d\phi}{dt} \right|^2
    + \norm{\V\chi}_{W^{2,2}(\Pi)}^2 \right) \\
    &\le a \left( \int_{\partial\Omega_0}\abs{\grad\V V}^2 \;dS
    + \left| \frac{d\phi}{dt} \right|^2
    + \norm{\V\chi}_{W^{2,2}(\Pi)}^2 \right) \\
    &\le a' \left( \norm{\grad\V V}_{W^{1,2}(\Omega_0)}^2
    + \left| \frac{d\phi}{dt} \right|^2
    + \norm{\V\chi}_{W^{2,2}(\Pi)}^2 \right),
  \end{align*}
  where $a$ and $a'$ are constants depending on $\rho$, $\mu$ and
  $\Omega$. Hence by \eqref{eq:carrierPext}$_1$ (and \cite[eq.~11]{MR2196495}),
  for some positive constant $c_g=c_g(\rho, \mu, \Omega)$,
  \begin{equation*}
    \norm{g}_{L^2_T} \le c_g \norm{\phi}_{W^{1,2}_T}\,.
  \end{equation*}
  When higher regularities in \eqref{eq:forcesRegularityStrong} are assumed (in
  the case of strong solutions), using a similar argument as above and
  \eqref{eq:carrierPext}$_{2,3}$ (and \cite[Remark.~1]{MR2196495}), we also
  have
  \begin{align*}
    &\norm{\V f}_{L^\infty(0,T;L^2(\Omega)} \le c'_f \norm{\phi}_{W^{2,2}_T}\,,
    && \norm{g}_{L^\infty_T} \le c'_g \norm{\phi}_{W^{2,2}_T}\,, \\
    &\norm{\frac{\partial\V f}{\partial t}}_{L^\infty(0,T;L^2(\Omega))} \le
    c''_f \norm{\phi}_{W^{3,2}_T}\,,
    && \norm{\frac{\partial g}{\partial t}}_{L^\infty_T} \le c''_g
    \norm{\phi}_{W^{3,2}_T}\,,
  \end{align*}
  where, again, positive constants above are at most functions of $\rho$,
  $\mu$, and $\Omega$.

  When the external body forces $\V{\tilde{f}}$ and $\tilde g$ are present,
  with the regularity assumed in \eqref{eq:forcesRegularityWeak} (and/or
  \eqref{eq:forcesRegularityStrong}), and $\V f$ and $g$ are given by
  \begin{equation} \tag{\ref{eq:forcePext}$'$} \label{eq:altforcePext}
    \begin{aligned}
      \V f(\V x,t) &= \frac{\mu}{\rho} \laplace\V V - \V V \cdot \grad\V V -
      \frac{\partial\V V}{\partial t} - \nabla \tilde p + \V{\tilde{f}}, \\
      g(t) &= \mu \int_\Gamma \V e_1 \cdot (\grad\V V + (\grad\V V)^T) \cdot \V
      n \,dS - \rho \int_\Gamma \tilde p n_1\,dS + \tilde g,
    \end{aligned}
  \end{equation}
  we have the following obvious modifications to the above estimates for the
  forces:
  \begin{align}
    \norm{\V f}_{L^2(0,T;L^2(\Omega))} &\le c_f\norm{\phi}_{W^{1,2}_T} +
    \norm{\V{\tilde{f}}}_{L^2(0,T;L^2(\Omega))}\,, \label{eq:fSmallnessBound}
    \\
    \norm{g}_{L^2_T} &\le c_g \norm{\phi}_{W^{1,2}_T}
    + \norm{\tilde g}_{L^2_T}\,, \label{eq:gSmallnessBound} \\
    \norm{\V f}_{L^\infty(0,T;L^2(\Omega))} &\le c'_f\norm{\phi}_{W^{2,2}_T} +
    \norm{\V{\tilde{f}}}_{L^\infty(0,T;L^2(\Omega))}\,, \\
    \norm{g}_{L^\infty_T} &\le c'_g \norm{\phi}_{W^{2,2}_T}
    + \norm{\tilde g}_{L^\infty_T}\,,  \\
    \norm{\frac{\partial\V f}{\partial t}}_{L^\infty(0,T;L^2(\Omega))} &\le
    c''_f \norm{\phi}_{W^{3,2}_T}
    + \norm{\frac{\partial\V{\tilde f}}{\partial
        t}}_{L^\infty(0,T;L^2(\Omega))}\,, \label{eq:dfdtSmallnessBound} \\
    \norm{\frac{\partial g}{\partial t}}_{L^\infty_T} &\le
    c''_g \norm{\phi}_{W^{3,2}_T}
    + \norm{\frac{\partial\tilde g}{\partial t}}_{L^\infty_T}\,.
    \label{eq:dgdtSmallnessBound} 
  \end{align}
\end{rmk}

\section{Weak Solutions} \label{sec:weak}

In this section we first give the definition of a weak solution to
\eqref{eq:noninertial0phi} and its equivalence to the original equations when
the weak solutions posses enough regularity and then we prove the existence of
such solutions along with the energy inequalities that they satisfy. 
\begin{dfn}
  A pair $(\V u(\V x, t), z(t))$ is called a $T$-periodic weak solution to
  \eqref{eq:noninertial}, corresponding to a flux $\phi(t) \in W^{1,2}_T$,
  with augmented $T$-periodic forces $\V{\tilde{f}} \in
  L^2(0,T;L^2(\Omega))$ and $\tilde{g} \in L^2([0,T])$ as in
  Remark \ref{rmk:forces}, if there is a $T$-periodic $\V V \in
  W^{1,2}(0,T;L^2_{loc}(\Omega)) \cap L^2(0,T;W^{2,2}_{loc}(\Omega))$ with
  $\int_\mathcal{S} \V V \cdot \V n_s \;dS = \phi(t)$, such that
  \begin{enumerate}
    \item $\V v = \V u - \V V \in L^\infty(0,T;\mathcal{D}) \cap
      L^2(0,T;\mathcal{D}^1)$ and $z \in W^{1,2}_T$,
  \item For all $\V\psi \in \mathcal{D}_0^\infty$ and $\eta \in
    C_T^\infty(\mathbb{R})$, with $\beta = \V e_1 \cdot \V\psi|_\Gamma$
    \begin{multline} \label{eq:weak1}
      \int_0^T \left\{(\V v,\V\psi) \frac{d\eta}{dt} - \left((\V v -
      \frac{dz}{dt}\V e_1) \cdot \grad\V v, \V\psi \right)\eta  -
      \frac{\mu}{\rho} (\V D(\V v), \V D(\V\psi)) \eta \right.\\\left. + \beta
      \left( \frac{\mm}{\rho} \frac{dz}{dt}\frac{d\eta}{dt} - \frac{\ks}{\rho}
      z\eta \right) \right\}dt \\= \int_0^T \left\{(\V V \cdot \grad\V v,
      \V\psi)\eta + \left((\V v - \frac{dz}{dt}\V e_1) \cdot \grad\V V,
      \V\psi\right)\eta - (\V f, \V\psi)\eta - \frac{\beta}{\rho} g\eta
      \right\}dt.
    \end{multline}
    In the above, $\V f$ and $g$ are given by \eqref{eq:altForces},
  \item For all scalar functions $\theta \in C_0^\infty(\Omega\cup\Gamma)$, and
    for almost all $t \in [0,T]$,
    \begin{equation} \label{eq:weak2}
      (\V v - \frac{dz}{dt}\V e_1, \grad\theta) = 0.
    \end{equation}
  \end{enumerate}
\end{dfn}
\begin{rmk} \label{rmk:weakToStrong}
  It is readily seen that if $(\V v, z)$ are smooth enough functions satisfying
  \eqref{eq:weak1} and \eqref{eq:weak2} for some $\V V$, then there is a
  function $p$ such that $(\V v, p, z)$ satisfies \eqref{eq:noninertial0phi}
  almost everywhere (in space and time) and hence $(\V v + \V V, p, z)$ will
  satisfy \eqref{eq:noninertial}. In fact, setting $\beta = 0$ in
  \eqref{eq:weak1}, integrating by parts and choosing $\eta \in
  C_T^\infty(\mathbb{R})$ such that $\eta(0)=\eta(T)=0$, we deduce
  \eqref{eq:noninertial0phi}$_1$ and then (by considering arbitrary $\eta$) the
  periodicity condition \eqref{eq:noninertial0phi}$_7$ for $\V v$. Using this
  information in \eqref{eq:weak1} (with $\beta \ne 0$ and after integrating by
  parts), we obtain \eqref{eq:noninertial0phi}$_3$ (with $\eta(0)=\eta(T)=0$)
  and the periodicity condition \eqref{eq:noninertial0phi}$_7$ for $z$ (with
  arbitrary $\eta$). Clearly, \eqref{eq:noninertial0phi}$_2$,
  \eqref{eq:noninertial0phi}$_4$ and \eqref{eq:noninertial0phi}$_6$ hold for
  any $\V v \in \mathcal{D}^1$. \eqref{eq:noninertial0phi}$_5$ follows from
  \eqref{eq:weak2} after integration integration by parts.
\end{rmk}
\begin{thm} \label{thm:weak}
  For any $T>0$, let $\phi(t) \in W^{1,2}_T$ be such that $\phi(t)$ satisfies
  the ``smallness condition'' \eqref{eq:smallnessWeak} below, then for any
  $T$-periodic forces $\V{\tilde{f}} \in L^2(0,T;L^2(\Omega))$
  and $\tilde{g} \in L^2([0,T])$, there exists at least one $T$-periodic weak
  solution to \eqref{eq:noninertial}.
\end{thm}
\begin{proof}
  We consider the flux carrier $\V V$, discussed in Section \ref{sec:flux}, and
  use Faedo-Galerkin approximations to find a solution $(\V v, z)$ to
  \eqref{eq:weak1} and \eqref{eq:weak2}. 

  Let $\{\V\psi_i\}_{i=1,2,\dots} \subset \mathcal{D}_0^\infty$ be a basis of
  $\mathcal{D}^1$ orthonormal in $\mathcal{D}$. We further assume, without loss
  of generality, that $\beta_1 = \V e_1 \cdot \V\psi_1|_\Gamma > 0$ and look
  for approximate solutions,
  \begin{equation*}
    \V v_n(\V x, t) = \sum_{i=1}^n a_n^i(t) \V\psi_i(\V x), \qquad
    \text{ and } \qquad z_n(t),
  \end{equation*}
  where, $a_n^i$ and $z_n$ are required to satisfy
  \begin{equation} \label{eq:GalerkinODE}
    \begin{aligned}
      &A_{i\ki}\frac{da_n^i}{dt} - c_{ij\ki} a_n^i a_n^j +
      \frac{\ks}{\rho} \beta_\ki z_n + (b_{i\ki} + d_{i\ki}) a_n^i = g_\ki +
      f_\ki, \\
      &\frac{dz_n}{dt} = \beta_i a_n^i, \\
      &a_n^i(t+T) = a_n^i(t), \qquad z_n(t+T) = z_n(t).
    \end{aligned}
  \end{equation}
  for $1 \le i,j,\ki \le n$ with summation on repeated indices. In the above,
  $\beta_i = \V e_1 \cdot \V\psi_i|_\Gamma$ and
  \begin{gather*}
    A_{i\ki} = \delta_{i\ki} + \frac{\mm}{\rho} \beta_i \beta_\ki, \qquad
    c_{ij\ki} = \left((\V\psi_i - \beta_i\V e_1) \cdot \grad\V\psi_j,
    \V\psi_\ki \right),\\
    b_{i\ki} = \frac{2\mu}{\rho}(\V D(\V\psi_i), \V D(\V\psi_\ki)), \qquad
    d_{i\ki} = (\V V \cdot \grad\V\psi_i, \V\psi_\ki) + ((\V\psi_i-\beta_i\V
    e_1) \cdot \grad\V V, \V\psi_\ki),\\
    g_\ki = \frac1\rho g \beta_\ki, \qquad f_\ki = (\V f, \V\psi_\ki).
  \end{gather*}
  
  To assert the existence of solutions to \eqref{eq:GalerkinODE},
  consider the following ``linearization'' of \eqref{eq:GalerkinODE}:
  \begin{equation} \label{eq:linearGalerkinODE}
    \begin{aligned}
      &A_{i\ki}\frac{da_n^i}{dt} - c_{ij\ki}
      \tilde{a}_n^i a_n^j + \frac{\ks}{\rho} \beta_\ki z_n + (b_{i\ki} +
      d_{i\ki}) a_n^i = g_\ki + f_\ki, \\
      &\frac{dz_n}{dt} = \beta_i a_n^i, \\
      &a_n^i(t+T) = a_n^i(t), \qquad z_n(t+T) = z_n(t).
    \end{aligned}
  \end{equation}
  where, $\tilde{a}_n^i \in L^2_T$, $1 \le i \le n$, are given $T$-periodic
  functions. The corresponding homogeneous system, with $f_\ki
  \equiv g_\ki \equiv d_{i\ki} \equiv 0$ (See Remarks \ref{rmk:fvHomo} and
  \ref{rmk:forces}), is
  \begin{equation} \label{eq:linearGalerkinODEhom}
    \begin{aligned}
      &A_{i\ki}\frac{da_n^i}{dt} - c_{ij\ki}
      \tilde{a}_n^i a_n^j + \frac{\ks}{\rho} \beta_\ki z_n + b_{i\ki}
      a_n^i = 0, \\
      &\frac{dz_n}{dt} = \beta_i a_n^i, \\
      &a_n^i(t+T) = a_n^i(t), \qquad z_n(t+T) = z_n(t),
    \end{aligned}
  \end{equation}
  and it has only the trivial solution $z_n = a_n^i = 0$, $1 \le i \le
  n$. This can be seen by multiplying the first equation above by $a_n^\ki$ and
  summing over $\ki$, then multiplying the second equation by $\ks z_n/\rho$
  and replacing in the first, to get
  \begin{equation*}
    \frac12 A_{i\ki}\frac{d(a_n^\ki a_n^i)}{dt} + \frac{\ks}{2\rho}
    \frac{dz_n}{dt} + b_{i\ki} a_n^\ki a_n^i = 0.
  \end{equation*}
  Integrating this equation over a period $T$ and using the periodicity
  conditions, it yields $\int_0^T b_{i\ki} a_n^\ki a_n^i \;dt = \int_0^T
  \norm{\V D(\V v_n)}^2 \;dt =0$; that is, by Remark \ref{rmk:DvGradv},
  $\V v_n = 0$ (and hence $a_n^i = 0,\;1 \le i \le n$). Replacing this
  information back in \eqref{eq:linearGalerkinODEhom}$_1$, we get
  $\beta_\ki z_n(t) = 0$, for all $1 \le \ki \le n$, and this, in view of our
  choice of basis with $\beta_1 \ne 0$, gives $z_n(t) = 0$.

  Consequently, \eqref{eq:linearGalerkinODE} has a unique $T$-periodic
  solution, $(a_n^i, z_n) \in (W^{1,2}_T)^n \times W^{2,2}_T$, for
  any given $f_\ki$, $g_\ki$ and $\tilde{a}_n^i$ in
  $L^2_T$, see \eg~\cite[Theorem 1.2.1]{MR2761514}. 

  Let $S_n = \mathrm{span}\{\V \psi_1, \dots, \V \psi_n\}$, the
  existence of $T$-periodic solutions to the nonlinear system
  \eqref{eq:GalerkinODE}, will be proven by showing the existence of a fixed
  point for the mapping
  \begin{equation} \label{eq:Phi}
    \Phi: L^2(0,T;S_n) \times W^{1,2}_T \longrightarrow W^{1,2}(0,T;S_n) \times
    W^{2,2}_T \subset L^2(0,T;S_n) \times W^{1,2}_T
  \end{equation}
  which maps any $(\tilde{\V v}_n, \tilde{z}_n)$ in its domain to the
  unique solution $(\V v_n, z_n)$ of \eqref{eq:linearGalerkinODE}. This will be
  asserted using the Leray-Schauder fixed point principle (see \eg~\cite[Theorem 6.A]{MR816732}):

  Let
  \begin{equation}\label{eq:fixedPointSet}
    F = \{ (\V v_n, z_n) \in L^2(0,T;S_n) \times W^{1,2}_T: (\V v_n, z_n) =
    \alpha \Phi(\V v_n, z_n), \; 0<\alpha<1\}
  \end{equation}
  We claim that $F$ is a bounded set. This will require the following
  ``natural'' and ``particular'' energy estimates for elements of $F$. To fix
  the ideas we will use the norm of $\mathcal{D}^1$ ($W^{1,2}(\Omega)$) on
  $S_n$, and to ease the notation we drop the subscript $n$ in what follows.

  In a completely standard manner, we get the following energy equation for
  $(\V v, z) \in F$
  \begin{equation*}
    \frac12 \frac{d}{dt}(\rho\norm{\V v}^2 + \mm \abs{\frac{dz}{dt}}^2 + \ks
    \abs{z}^2) + 2\mu\norm{\V D(\V v)}^2 = -\rho((\V v -
    \frac{dz}{dt}\V e_1) \cdot \grad\V V, \V v) + \alpha\rho(\V f, \V v) +
    \alpha g \frac{dz}{dt}.
  \end{equation*}
  Using Corollary~\ref{cor:nonlinearBound} and Remark~\ref{rmk:DvGradv} we get
  \begin{equation*}
    \frac12 \frac{d}{dt}(\rho\norm{\V v}^2 + \mm \abs{\frac{dz}{dt}}^2 + \ks
    \abs{z}^2) + \mu\norm{\grad\V v}^2 \le \rho c_q \norm{\phi}_{W^{1,2}_T}
    \norm{\grad\V v}^2 + \alpha\rho(\V f, \V v) +
    \alpha g \frac{dz}{dt}.
  \end{equation*}
  So if
  \begin{equation} \label{eq:smallnessWeak}
    \stepcounter{equation} \tag{\theequation$\epsilon$}
    \norm{\phi}_{W^{1,2}_T} < \frac{\mu}{\rho c_q},
  \end{equation}
  using the boundary trace inequalities,
  \begin{equation*}
    \abs{\Gamma} \abs{\frac{dz}{dt}} = \norm{\frac{dz}{dt}\V
      e_1}_{L^2(\Gamma)} \le c'_b \norm{\V v}_{W^{1,2}(\Omega_r)} \le c_b
    \norm{\grad\V v},
  \end{equation*}
  we arrive at the following ``natural'' energy inequality
  \begin{equation} \label{eq:E}
    \frac{d}{dt}(\rho\norm{\V v}^2 + \mm \abs{\frac{dz}{dt}}^2 + \ks \abs{z}^2)
    + c_1 (\norm{\grad\V v}^2 + \abs{\frac{dz}{dt}}^2) \le
    c_2 (\norm{\V f}^2 + \abs{g}^2).
  \end{equation}
  Where $c_i$'s, in the above and in what follows, are constants depending at
  most on $\Omega$, $\mm$, $\ks$, $\rho$, $\mu$ and $\phi$. Specifically,
  integrating the above in $[0,T]$ over a period gives
  \begin{equation}\label{eq:partialBound}
    \int_0^T \norm{\grad\V v}^2 \,dt + \int_0^T
    \abs{\frac{dz}{dt}}^2 \,dt \le
    c_3 \int_0^T (\norm{\V f}^2_{L^2(\Omega)} + \abs{g}^2) \,dt
  \end{equation}

  Let
  \begin{equation}\label{eq:Efunc}
    E(t) = \mathcal{E}(\V v(t), \frac{dz(t)}{dt}, z(t)) \mathrel{\mathop:}=
    \frac12 (\rho\norm{\V v(t)}^2 + \mm \abs{\frac{dz(t)}{dt}}^2 + \ks
    \abs{z(t)}^2).
  \end{equation}
  Clearly, \eqref{eq:E} is only ``partially'' dissipative in $E$, as there is
  no contribution from $\abs{z}$ in the dissipation term. We next show that, by
  choosing a suitable equivalent energy functional, it is possible to obtain an
  energy relation with complete dissipation and conclude the boundedness of
  $F$.

  From \eqref{eq:linearGalerkinODE} with $\ki = 1$, we have that $(\V v, z) \in
  F$ satisfies
  \begin{multline} \label{eq:helper1}
    \rho \frac{d(\V v, \V\psi_1 )}{dt} + \mm\beta_1\frac{d^2z}{dt^2} +
    \ks\beta_1 z = - 2 \mu (\V D(\V v), \V D(\V\psi_1)) 
    \\+ \rho [((\V v-\frac{dz}{dt}\V e_1) \cdot \grad\V v,
    \V\psi_1) - ((\V v-\frac{dz}{dt}\V e_1) \cdot \grad\V V, \V\psi_1) - (\V V
    \cdot \grad\V v, \V\psi_1)] \\+ \alpha \beta_1 g + \alpha \rho (\V f,
    \V\psi_1).
  \end{multline}
  Multiplying the above by $z$, and using H\"older's inequality and Sobolev
  embedding theorem, we obtain
  \begin{align}
    \frac{d}{dt}[\rho z(\V v, \V\psi_1) + \mm\beta_1z\frac{dz}{dt}] &+
    \ks\beta_1 \abs{z}^2 \le \rho \abs{\frac{dz}{dt}}(\V v, \V\psi_1) +
    \mm\beta_1 \abs{\frac{dz}{dt}}^2 \nonumber \\
    &\hphantom{\le} + \rho (\norm{\grad \V v}
    +\norm{\grad\V V}+\abs{\frac{d z}{dt}}+2\frac\mu\rho) \abs{z}{\norm{\grad\V
        v}}\norm{\grad\V\psi_1} \nonumber \\
    &\hphantom{\le} + \rho (\norm{\grad\V v} +
    \abs{\frac{dz}{dt}}) \abs{z}\norm{\grad\V V} \norm{\grad\V\psi_1} \nonumber
    \\ \label{eq:helper2}
    &\hphantom{\le} + \beta_1 \abs{g}\abs{z} + \rho \abs{z} \norm{f}
    \norm{\V\psi_1}.
  \end{align}
  For $\delta \le
  \min\{1,\displaystyle\frac{1}{\norm{\V\psi_1}},\frac1\beta_1,
  \frac{\ks}{\rho\norm{\V\psi_1}+\mm\beta_1}\}$
  \begin{multline}\label{eq:Gfunc}
    E(t) \le G(t) = \mathcal{G}_\delta^{\psi_1}(\V v(t), \frac{dz(t)}{dt},
    z(t))\mathrel{\mathop:}= \rho\norm{\V v(t)}^2 + \mm
    \abs{\frac{dz(t)}{dt}}^2 + \ks \abs{z(t)}^2 \\+ \delta \rho z(t) (\V
    v(t),\V\psi_1) + \delta\mm\beta_1 z(t) \frac{dz(t)}{dt} \le 3 E(t),
  \end{multline}
  so $\mathcal{G}$ is an equivalent energy functional to
  $\mathcal{E}$. Multiplying \eqref{eq:helper2} by $\delta$ and adding to
  \eqref{eq:E}, with different estimates (compared to what is used above) for
  the terms $\alpha \rho(\V f, \V v)$ and $\alpha g(dz/dt)$, we obtain
  \begin{equation} \label{eq:Gunestimated}
    \begin{aligned}
    \frac{dG}{dt} + c_4 G &\le \rho
    \abs{\frac{dz}{dt}}\norm{\V v}\norm{\V\psi_1} + \mm\beta_1
    \abs{\frac{dz}{dt}} \abs{\frac{dz}{dt}} \\
    &\hphantom{\le} + \rho (\norm{\grad \V v}
    +\norm{\grad\V V}+\abs{\frac{dz}{dt}}+2\frac\mu\rho) \abs{z}{\norm{\grad\V
        v}}\norm{\grad\V\psi_1} \\
    &\hphantom{\le} + \rho (\norm{\grad\V v} +
    \abs{\frac{dz}{dt}}) \abs{z}\norm{\grad\V V} \norm{\grad\V\psi_1} \\
    &\hphantom{\le} + \beta_1 \abs{g}\abs{z} + \rho \abs{z} \norm{\V f}
    \norm{\V\psi_1} + \rho\norm{\V f} \norm{\V v} + \abs{g}\abs{\frac{dz}{dt}},
    \end{aligned}
  \end{equation}
  If $G(t_0) = 0$ for some $t_0 \in [0,T]$, we
  integrate \eqref{eq:E} in $[t_0,t]$, $t_0 \le t \le t_0+T$ with $E(t_0)=0$
  to deduce \eqref{eq:Fbound} below, otherwise, dividing the above equation by
  $\sqrt{G}$ and using Young's inequality, we get the following ``particular''
  energy inequality
  \begin{equation} \label{eq:G}
    \frac{d\sqrt{G}}{dt} + c_5 \sqrt{G} \le C_1
    (\norm{\grad\V v}^2 + \abs{\frac{dz}{dt}}^2) + c_6 (\norm{\grad\V V}^2 +
    \norm{\V f}^2 + \abs{g}^2) + C_2,
  \end{equation}
  where $C_i$'s, in the above and in what follows, depend also on $\V\psi_1$
  and $T$ (and $\V{\tilde{\V f}}$ and $\tilde{g}$ if non-zero) in addition to
  $\Omega$, $\mm$, $\ks$, $\rho$, $\mu$ and $\phi$. Integrating the above in
  $[0,T]$ and using \eqref{eq:partialBound}, \eqref{eq:carrierP} and
  \eqref{eq:carrierPext} gives
  \begin{equation} \label{eq:intG}
    \int_0^T \sqrt{G} \;dt \le C_3 \int_0^T (\norm{\grad\V
      V}_{L^2(\supp\V\psi_1)}^2 + \norm{\V f}_{L^2(\Omega)}^2 + \abs{g}^2)
    \;dt = C_4.
  \end{equation}

  Since $G$ is (absolutely) continuous, there is $t_0\in[0,T]$ such that
  $T\sqrt{G(t_0)} = \int_0^T \sqrt{G} \;dt \le C_4$, so integrating
  \eqref{eq:G} in $[t_0,t]$ for $t<T$, we obtain
  \begin{equation*}
    \sup_{t\in[0,T]} \sqrt{G(t)} \le C_4(1+\frac1T),
  \end{equation*}
  in particular,
  \begin{equation}\label{eq:Fbound}
    \sup_{t\in[0,T]} E(t) < C_5,
  \end{equation}
  and thus by \eqref{eq:partialBound}, the set $F$ is bounded in $L^2(0,T;S_n)
  \times W^{1,2}_T$:
  \begin{equation*}
    \int_0^T (\norm{\grad\V v}^2+\abs{z}^2+\abs{\frac{dz}{dt}}^2) \;dt \le
    C_6, \qquad \forall (\V v, z) \in F.
  \end{equation*}

  Let $(\V v, z) = \Phi(\V{\tilde{v}}, \tilde{z})$. In a totally similar
  manner as demonstrated above, we obtain the natural energy \eqref{eq:E} and
  the particular energy
  \begin{equation*}
    \frac{d\sqrt{G}}{dt} + c_5 \sqrt{G} \le C_7
    (\norm{\grad\V v}^2 +\abs{\frac{dz}{dt}}^2)+ C_8(\norm{\grad\tilde{\V
        v}}^2 + \abs{\frac{d\tilde z}{dt}}^2) + c_7 (\norm{\grad\V V}^2 +
    \norm{\V f}^2 + \abs{g}^2) + C_9,
  \end{equation*}
  so bounded sets
  \begin{equation*}
    \left\{(\V{\tilde{v}}, \tilde{z}) \in L^2(0,T;S_n) \times W^{1,2}_T:
    \int_0^T(\norm{\grad\V{\tilde v}}^2+\abs{\tilde z}^2+\abs{\frac{d\tilde
        z}{dt}}^2) \;dt \le c, \; c \in \mathbb{R}\right\},
  \end{equation*}
  are, indeed, mapped to uniformly bounded and equicontinuous sets.
  By the Ascoli-Arzel\`a theorem, $\Phi$ maps bounded sets into
  relatively compact sets. Moreover, for $(\V v_1, z_1) =
  \Phi(\V{\tilde{v}}_1,\tilde{z}_1)$ and $(\V v_2, z_2) =
  \Phi(\V{\tilde{v}}_2,\tilde{z}_2)$, with 
  \begin{gather*}
    E_{1-2}(t) = \mathcal{E}(\V v_1-\V v_2, \frac{dz_1}{dt}-\frac{dz_2}{dt},
    z_1-z_2), \\
    G_{1-2}(t) = \mathcal{G}^{\psi_1}_\delta(\V v_1-\V v_2,
    \frac{dz_1}{dt}-\frac{dz_2}{dt}, z_1-z_2),
  \end{gather*}
  we have
  \begin{equation*}
    \int_0^T ( \norm{\grad\V v_1 - \grad\V v_2}^2 +
    \abs{\frac{dz_1}{dt}-\frac{dz_2}{dt}}^2) \,dt \le C_{10}
    \int_0^T \norm{\grad\V{\tilde v}_1 - \grad\V{\tilde v}_2}^2 \,dt,
  \end{equation*}
  and
  \begin{multline*}
    \int_0^T \sqrt{G_{1-2}} \,dt \le C_{11}
    \int_0^T (\norm{\grad\V v_1 - \grad\V v_2}^2 + \abs{\frac{dz_1}{dt} -
      \frac{dz_2}{dt}}^2) \,dt \\+ C_{12} \int_0^T \norm{\grad\V{\tilde v}_1 -
      \grad\V{\tilde v}_2}^2 \,dt.
  \end{multline*}
  Hence, $\Phi$ is also continuous (and hence compact), and has a fixed
  point by Schaefer's fixed-point theorem. Considering the index $n$, that we
  had earlier suppressed to ease the notation, the fixed point, which (with an
  abuse of notation) we again denote by $(\V v_n, z_n)$, satisfies
  \eqref{eq:GalerkinODE}. It follows that for any $1\le i\le n$ (and $\eta \in
  C_T^\infty(\mathbb{R})$), $(\V v_n, z_n)$ satisfies:
  \begin{multline} \label{eq:boundedWeak}
    \int_0^T \left\{(\V v_n,\V\psi_i) \frac{d\eta}{dt} - \left((\V v_n -
    \frac{dz_n}{dt}\V e_1) \cdot \grad\V v_n, \V\psi_i \right)\eta -
    \frac{\mu}{\rho} (\V D(\V v_n), \V D(\V\psi_i)) \eta \right.\\\left. +
    \beta_i \left( \frac{\mm}{\rho} \frac{dz_n}{dt}\frac{d\eta}{dt} -
    \frac{\ks}{\rho} z_n\eta \right) \right\}dt \\= \int_0^T \left\{(\V V
    \cdot \grad\V v_n, \V\psi_i)\eta + \left((\V v_n - \frac{dz_n}{dt}\V
    e_1) \cdot \grad\V V, \V\psi_i\right)\eta - (\V f, \V\psi_i)\eta -
    \frac{\beta_i}{\rho} g\eta \right\}dt.
  \end{multline}

  Using \eqref{eq:partialBound} and \eqref{eq:Fbound}, we conclude the
  existence of $T$-periodic functions $(\V v, z)$ and a
  subsequence $\{(\V v_{n_k}, z_{n_k})\}_{k=1,2,\dots}$ such that
  {\renewcommand{\theenumi}{(\alph{enumi})}
   \renewcommand{\labelenumi}{\theenumi}
   \begin{enumerate}
    \item $\V v_{n_k}$ converges weekly to $\V v$ in
      $L^2(0,T;\mathcal{D})$.\label{itm:L2L2Bound}
    \item $\V v_{n_k}$ converges weekly to $\V v$ in
      $L^2(0,T;\mathcal{D}^1)$. \label{itm:L2H1Bound}
    \item $\V v_{n_k}$ converges weekly-\textasteriskcentered~ to $\V v$
      in $L^\infty(0,T;\mathcal{D})$. \label{itm:C0L2Bound} 
    \item $z_{n_k}$ converges weekly to $z$ in $W^{1,2}_T$. 
      \label{itm:dzdtBound}
  \end{enumerate}
  For a bounded subset $\Omega' \subset \Omega$, from item \ref{itm:L2L2Bound}
  and \eqref{eq:partialBound} together with \cite[Lemma II.5.2]{MR2808162},
  it follows that
  \begin{enumerate}
    \setcounter{enumi}{4}
    \item $\V v_{n_k}$ converges strongly to $\V v$ in
      $L^2(0,T;L^2(\Omega'))$. \label{itm:L2Strong}
  \end{enumerate}
  }
  \noindent The above convergences allow taking the limit along the
  subsequence $(\V v_{n_k}, z_{n_k})$ as $k\rightarrow\infty$ of
  \eqref{eq:boundedWeak} to obtain, for all $i \ge 1$ :
  \begin{multline*}
    \int_0^T \left\{(\V v,\V\psi_i) \frac{d\eta}{dt} - \left((\V v -
    \frac{dz}{dt}\V e_1) \cdot \grad\V v, \V\psi_i \right)\eta  -
    \frac{\mu}{\rho} (\V D(\V v), \V D(\V\psi_i)) \eta \right.\\\left. +
    \beta_i \left( \frac{\mm}{\rho} \frac{dz}{dt}\frac{d\eta}{dt} -
    \frac{\ks}{\rho} z\eta \right) \right\}dt = \int_0^T \left\{(\V V
    \cdot \grad\V v, \V\psi_i)\eta \vphantom{\frac{d}{d}}\right.\\\left. +
    \left((\V v - \frac{dz}{dt}\V e_1) \cdot \grad\V V, \V\psi_i\right)\eta -
    (\V f, \V\psi_i)\eta - \frac{\beta_i}{\rho} g\eta \right\}dt.
  \end{multline*}
  \eqref{eq:weak1} follows from the above by a simple density argument.

  To obtain \eqref{eq:weak2}, we note that for any $\theta \in
  C_0^\infty(\Omega\cup\Gamma)$ and an arbitrary $\eta \in L^2_T$, it follows
  from \eqref{eq:GalerkinODE}$_2$ that
  \begin{equation*}
    \int_0^T (\V v_n - \frac{dz_n}{dt}\V e_1, \grad\theta)\, \eta \,dt = 0.
  \end{equation*}
  \eqref{eq:weak2} follows from the above and the convergences
  \ref{itm:L2L2Bound} and \ref{itm:dzdtBound}, and this completes the proof.
\end{proof}
\begin{rmk} \label{rmk:smallnessEnergy}
  For the purpose of strong solutions later, we shall note that estimating
  the terms in \eqref{eq:Gunestimated} differently we get the following
  analogous inequality instead of \eqref{eq:G}:
  \begin{equation*}
    \frac{d\sqrt{G}}{dt} + c_5 \sqrt{G} \le C'_1
    (\norm{\grad\V v}^2 + \abs{\frac{dz}{dt}}^2 + \norm{\grad\V v} +
    \abs{\frac{dz}{dt}}) + c'_6 (\norm{\grad\V V}^2 +
    \norm{\V f} + \abs{g}),
  \end{equation*}
  and so \eqref{eq:intG} will read
  \begin{equation*}
    \int_0^T \sqrt{G} \;dt \le C'_4,
  \end{equation*}
  where $C'_4$ can be made as small as we wish by taking
  $\norm{\phi}_{W^{1,2}_T}$, $\norm{\V{\tilde{f}}}_{L^2(0,T;L^2(\Omega))}$
  and $\abs{\tilde{g}}_{L^2_T}$ sufficiently small.
\end{rmk}
\begin{rmk}  
  The week solutions $(\V v, z)$ obtained above satisfy an ``energy
  inequality.'' Indeed, taking the limit inferior (as
  $k\rightarrow\infty$) of \eqref{eq:partialBound} and using the convergences
  in items \ref{itm:L2H1Bound} and \ref{itm:dzdtBound} above, we obtain:
  \begin{equation} \label{eq:partialEnergy}
    \int_0^T \norm{\grad\V v}^2 \,dt + \int_0^T \abs{\frac{dz}{dt}}^2 \,dt \le
    c_3 \int_0^T \norm{\V f}^2 + \abs{g}^2) \,dt.
  \end{equation}
  Analogously, from \eqref{eq:Fbound} it follows that for every
  non-negative $\theta(t)$, we have
  \begin{equation*}
    \int_0^T (\norm{\V v_{n_k}}^2 + \abs{\frac{dz_{n_k}}{dt}}^2 +
    \abs{z_{n_k}}^2) \, \theta(t) \,dt \le \int_0^T C_{13} \theta(t) \,dt.
  \end{equation*}
  Again, taking the limit inferior of the above and noting that $\theta(t)$ is
  an arbitrary non-negative function, we get the estimate,
  \begin{equation} \label{eq:energy}
    \esssup_{t\in[0,T]} \left( \norm{\V v}^2 + \abs{\frac{dz}{dt}}^2 +
    \abs{z}^2 \right) \le C_{13}.
  \end{equation}

  So, in fact, the weak solution $(\V v, z)$ is such that $z \in
  W^{1,\infty}_T$.
\end{rmk}
\begin{rmk} \label{rmk:assymptotics}
  It follows from \cite[Proposition 1.]{MR2196495} that the weak solution for
  the fluid, obtained in Theorem \ref{thm:weak}, converges to $\V \chi$ as
  $\abs{\V x}\rightarrow \infty$, in a weak sense:
  \begin{equation*}
    \lim_{X\rightarrow\infty} \norm{\V v}_{L^2(0,T;L^3(\Omega^{\pm X}))} = 0.
  \end{equation*}
\end{rmk}

\section{Strong Solutions}

In this section we show that if the flow rate, $\phi(t)$, and the external
forces, $\V{\tilde{f}}$ and $\tilde g$, are ``small'' and more ``regular'', the
solution to \eqref{eq:noninertial} constructed in the previous section is more
regular. We note that for the existence of weak solutions of the previous
section there is no ``smallness'' condition needed on $\V{\tilde{f}}$ and
$\tilde g$. The smallness condition(s) will be given in the following Lemma and
for the regularity we assume \eqref{eq:forcesRegularityStrong}.
\begin{lem} \label{lem:timeRegularity}
  Let $T>0$. Assume that the $T$-periodic flow rate $\phi(t) \in W^{3,2}_T$,
  and the $T$-periodic forces $\V{\tilde{f}} \in W^{1,\infty}(0,T;L^2(\Omega))$
  and $\tilde{g} \in W^{1,\infty}_T$ satisfy the smallness conditions
  \eqref{eq:smallnessWeak} (given in the proof of Theorem \ref{thm:weak}) and
  \eqref{eq:smallnessStrong1} and \eqref{eq:smallnessStrong2} given
  below. Then, there is a weak solution $(\V u, z)$ to
  \eqref{eq:noninertial0phi} that satisfies \eqref{eq:weak1} with $\V V$ being
  the same flux carrier constructed before and $\V f$ and $g$ given by
  \eqref{eq:altforcePext}. Moreover, $(\V v, z)$ (with $\V v = \V u - \V V$)
  belongs to the following regularity class
  \begin{align*}
    \V v \in W^{1,\infty}(0,T;\mathcal{D}) \cap W^{1,2}(0,T;\mathcal{D}^1), &&
    z \in W^{2,\infty}_T.
  \end{align*}
\end{lem}
\begin{proof}
  The proof follows closely the proof of higher regularity for the solutions to
  the Navier-Stokes initial boundary value problem, given in \cite[Theorem
    3.3.7]{MR1846644}. To put the (time-) periodic solutions at hand to an
  initial boundary value problem setting, we note that by
  \eqref{eq:partialBound}, we can choose the forces small enough such that at
  each Galerkin approximation level, $n$, for any $\epsilon > 0$, we have a
  $t_n^*$ such that
  \begin{equation} \label{eq:smallInitialGradv}
    \norm{\grad\V v_n(t_n^*)}^2 + \abs{\frac{dz}{dt}(t_n^*)}^2 < \epsilon,
  \end{equation}
  In fact, by \eqref{eq:partialBound}, \eqref{eq:fSmallnessBound} and
  \eqref{eq:gSmallnessBound}, this will be the case when
  \begin{equation*}
    2c_3\left((c_f^2+c_g^2)\norm{\phi}_{W^{1,2}_T}^2 +
      \norm{\V{\tilde{f}}}_{L^2(0,T;L^2(\Omega))}^2 +
      \abs{\tilde{g}}_{L^2_T}^2\right) < \epsilon T.
  \end{equation*}

  Differentiating \eqref{eq:GalerkinODE} with respect to $t$, then multiplying
  the resulting equation by $da^\ki_n/dt$ and summing over $\ki$, we get
  \begin{multline*}
    \frac12 \frac{d\norm{\V v'_n}^2}{dt} + \frac{\mm}{2\rho}
    \frac{d\abs{z''_n}^2}{dt} + \frac{2\mu}{\rho} (\V D(\V v'_n), \V D(\V
    v'_n)) = -\frac{\ks}{2\rho} z'_n z''_n + ((\V v'_n - z''_n\V e_1)
    \cdot \grad\V v_n, \V v'_n) -\\ (\V V' \cdot \grad\V v_n, \V v'_n) -
    ((\V v'_n - z''_n\V e_1) \cdot \grad\V V, \V v'_n) - ((\V v_n - z'_n\V e_1)
    \cdot \grad\V V', \V v'_n) +\\ \frac1\rho g'z''_n + (\V f', \V v'_n).
  \end{multline*}
  Where, to ease the notation, we have used $'$ to denote differentiation with
  respect to $t$. With the help of Lemma \ref{lem:nonlinearBound}, H\"older,
  Young and Poincar\'e inequalities 
  \begin{multline*}
    \frac12 \frac{d\norm{\V v'_n}^2}{dt} + \frac{\mm}{2\rho}
    \frac{d\abs{z''_n}^2}{dt} + \frac{2\mu}{\rho} (\V D(\V v'_n), \V D(\V
    v'_n)) \le \\\frac{\ks}{2\rho} \abs{z'_n}\,\abs{z''_n} + c_8 \norm{\grad\V
      v_n} (\norm{\grad v'_n}^2 + \abs{z''_n}^2) - (\V V' \cdot \grad\V v_n,
    \V v'_n) - \\c_q \norm{\phi}_{W^{1,2}_T} \norm{\grad\V v'_n}^2 - ((\V v_n -
    z'_n\V e_1) \cdot \grad\V V', \V v'_n) + \frac1\rho g'z''_n + (\V f', \V
    v'_n).
  \end{multline*}
  In a similar manner as the proof of Lemma \ref{lem:nonlinearBound} we have
  \begin{align*}
    (\V V' \cdot \grad\V v_n, \V v'_n) &\le c'_q \norm{\phi}_{W^{2,2}_T}
    \norm{\grad\V v_n} \,\norm{\grad\V v'_n}, \\
    ((\V v_n - z'_n\V e_1) \cdot \grad\V V', \V v'_n) & \le c'_q
    \norm{\phi}_{W^{2,2}_T} \norm{\grad\V v_n} \,\norm{\grad\V v'_n},
  \end{align*}
  where $c'_q = c'_q(\Omega, \mu, \rho)$. So by \eqref{eq:smallnessWeak} and
  Young's, H\"older and Poincar\'e inequalities (and boundary trace inequality
  for $\V v'_n$), we get
  \begin{multline} \label{eq:dvdtEnergy}
    \frac{d}{dt}(\norm{\V v'_n}^2 + \frac{\mm}{\rho}\abs{z''_n}^2) +
        (c_{10} - c_9 \norm{\grad\V v_n} - c_8 \norm{\grad\V v_n}^2)
    (\norm{\grad\V v'_n}^2 + \frac{\mm}{\rho}\abs{z''_n}^2) \le c_{11}
    \abs{z'_n}^2 + \\
    c_{12} (\norm{\phi}^2_{W^{2,2}_T} + \abs{g'}^2 + \norm{\V f'}^2).
  \end{multline}
  Choosing $\displaystyle\epsilon <
  \left(\frac{c_9-\sqrt{c_9^2+4c_{10}c_8}}{-2c_8}\right)^2$ in
    \eqref{eq:smallInitialGradv}, that is requiring that
  \begin{equation} \label{eq:smallnessStrong1}
    \stepcounter{equation} \tag{\theequation$\epsilon$}
    2c_3\left((c_f^2+c_g^2)\norm{\phi}_{W^{1,2}_T}^2 +
    \norm{\V{\tilde{f}}}_{L^2(0,T;L^2(\Omega))}^2 +
    \abs{\tilde{g}}_{L^2_T}^2\right) <
    \left(\frac{c_9-\sqrt{c_9^2+4c_{10}c_8}}{-2c_8}\right)^2 T,
  \end{equation}
  we deduce that in an interval containing $t^*_n$, the coefficient $c_{10} -
  c_9 \norm{\grad\V v_n(t)} - c_8 \norm{\grad\V v_n(t)}^2$ is positive. If this
  is the case in $[t^*_n, t^*_n+T]$ then we have \eqref{eq:positiveCoeff}
  below, and we continue the argument from there. Otherwise, there is a time
  $\bar{t}_n$, where
  \begin{equation} \label{eq:positiveCoeffInterval}
    \begin{aligned}
      c_{10} - c_9 \norm{\grad\V v_n(t)} - c_8 \norm{\grad\V v_n(t)}^2 &> 0,
      \qquad \text{ for } t^*_n \le t < \bar{t}_n < t^*_n+T, \\
      c_{10} - c_9 \norm{\grad\V v_n(\bar{t}_n)} - c_8 \norm{\grad\V
        v_n(\bar{t}_n)}^2 &= 0.
    \end{aligned}
  \end{equation}
  Integrating \eqref{eq:dvdtEnergy} in $(t^*_n,\bar{t}_n)$, we have, by
  \eqref{eq:Fbound} (and the Poincar\'e inequality)
  \begin{multline} \label{eq:primeBounds}
    \norm{\V v'_n(\bar{t}_n)}^2 + \frac{\mm}{\rho}\abs{z''_n(\bar{t}_n)}^2 \le
    \\
    e^{\displaystyle\int_{t^*_n}^{\bar{t}_n} c_{13}(c_8\norm{\grad\V v_n}^2 +
      c_9\norm{\grad\V v_n} - c_{10}) dt} (\norm{\V v'_n(t^*_n)}^2 +
    \frac{\mm}{\rho}\abs{z''_n(t^*_n)}^2)
    \\
    + \left(C_{14} + c_{12} (\norm{\phi}^2_{W^{2,2}_T} +
    \norm{g'}_{L^\infty_T} + \norm{\V
      f'}_{L^\infty(0,T;L^2(\Omega))})\right) (\bar{t}_n - t^*_n).
  \end{multline}
  Then we have the following two cases:
  
  \textit{Case I:} $\norm{\V v'_n(\bar{t}_n)}^2 +
  \frac{\mm}{\rho}\abs{z''_n(\bar{t}_n)}^2 \ge \norm{\V v'_n(t^*_n)}^2 +
  \frac{\mm}{\rho}\abs{z''_n(t^*_n)}^2$. Hence from \eqref{eq:primeBounds}, we
  get that
  \begin{multline} \label{eq:primeBoundsHelper}
    \norm{\V v'_n(\bar{t}_n)}^2 + \frac{\mm}{\rho}\abs{z''_n(\bar{t}_n)}^2 \le
    C_{15}\left(C_{14} + c_{12} (\norm{\phi}^2_{W^{2,2}_T} +
    \norm{g'}_{L^\infty_T} + \norm{\V
      f'}_{L^\infty(0,T;L^2(\Omega))})\right),
  \end{multline}
  where
  \begin{equation*}
    C_{15} = \left|\frac{\bar{t}_n -
        t^*_n}{1-e^{\displaystyle\int_{t^*_n}^{\bar{t}_n}
        c_{13}(c_8\norm{\grad\V v_n}^2 + c_9\norm{\grad\V v_n} - c_{10})
        dt}}\right|.
  \end{equation*}
  Note that $C_{15}$ is bounded in view of \eqref{eq:positiveCoeffInterval}
  and $C_{14}$ can be made as small as we wish by choosing
  $\norm{\phi}_{W^{1,2}_T}$, $\norm{\V{\tilde{f}}}_{L^2(0,T;L^2(\Omega))}$ and
  $\abs{\tilde{g}}_{L^2_T}$ sufficiently small, by Remark
  \ref{rmk:smallnessEnergy}.

  From \eqref{eq:E}, written at $t = \bar{t}_n$ and using Young's inequality,
  \eqref{eq:Fbound} and \eqref{eq:primeBoundsHelper} we get
  \begin{align*}
    \norm{\grad\V v_n(\bar{t}_n)}^2 + \abs{\frac{dz}{dt}(\bar{t}_n)}^2 &\le
    c_{14} (\norm{\V f}^2_{L^\infty(0,T;L^2(\Omega))} +
    \norm{g}^2_{L^\infty_T}) \\
    &\hphantom{\le} + c_{15}(\norm{\V v'_n(\bar{t}_n)}^2 +
    \frac{\mm}{\rho} \abs{z''_n(\bar{t}_n)}^2) + C_{16} \\
    & \le c_{14} (\norm{\V f}^2_{L^\infty(0,T;L^2(\Omega))} +
    \norm{g}^2_{L^\infty_T}) \\
    &\hphantom{\le} + c_{15} C_{15}\left(C_{14} + c_{12}
    (\norm{\phi}^2_{W^{2,2}_T} + \norm{g'}_{L^\infty_T} \right. \\
    &\hphantom{\le + c_{14} C_{15}(C_{14} + c_{12}
    (\norm{\phi}^2_{W^{2,2}_T}} \left. + \norm{\V
      f'}_{L^\infty(0,T;L^2(\Omega))})\right) + C_{16},
  \end{align*}
  where, again, in view of Remark \ref{rmk:smallnessEnergy}, we note that
  $C_{16}$ can be made small by choosing suitable norms of the forces small.
  So if, in addition to \eqref{eq:smallnessStrong1}, we also require that
  \begin{multline} \label{eq:smallnessStrong2}
    \stepcounter{equation} \tag{\theequation$\epsilon$}
    c_{14} (\norm{\V f}_{L^\infty(0,T;L^2(\Omega))} +
    \norm{g}^2_{L^\infty_T}) \\+ c_{15} C_{15}\left(C_{14} + c_{12}
    (\norm{\phi}^2_{W^{2,2}_T} + \norm{g'}_{L^\infty_T}
    + \norm{\V f'}_{L^\infty(0,T;L^2(\Omega))})\right) \\+ C_{16} <
    \left(\frac{c_9-\sqrt{c_9^2+4c_{10}c_8}}{-2c_8}\right)^2,
  \end{multline}
  then we have that $c_{10} - c_9 \norm{\grad\V v_n(\bar{t}_n)} - c_8
  \norm{\grad\V v_n(\bar{t}_n)}^2 > 0$, which is in contradiction with
  \eqref{eq:positiveCoeffInterval}$_2$ and, in fact, we have
  \begin{equation} \label{eq:positiveCoeff}
    c_{10} - c_9 \norm{\grad\V v_n(t)} - c_8 \norm{\grad\V
      v_n(t)}^2 > \delta', \qquad \forall t \in \mathbb{R},
  \end{equation}
  where $\delta'>0$ does not depend on $n$ and is determined by the
  ``smallness'' conditions \eqref{eq:smallnessWeak},
  \eqref{eq:smallnessStrong1} and \eqref{eq:smallnessStrong2}. Then from
  \eqref{eq:dvdtEnergy} we have
  \begin{equation*}
    \norm{\V v'_n(t)}^2 + \frac{\mm}{\rho}\abs{z''_n(t)}^2 + \delta'
    \int_0^T(\norm{\grad\V v'_n}^2 + \frac{\mm}{\rho}\abs{z''_n}^2) \le C_{17},
    \qquad 0\le t \le T,
  \end{equation*}
  and the claim of the Lemma follows in this case.

  \textit{Case II:} $\norm{\V v'_n(\bar{t}_n)}^2 +
  \frac{\mm}{\rho}\abs{z''_n(\bar{t}_n)}^2 < \norm{\V v'_n(t^*_n)}^2 +
  \frac{\mm}{\rho}\abs{z''_n(t^*_n)}^2$. Hence from \eqref{eq:primeBounds}, we
  get that
  \begin{multline*}
    \norm{\V v'_n(\bar{t}_n)}^2 + \frac{\mm}{\rho}\abs{z''_n(\bar{t}_n)}^2 <
    \norm{\V v'_n(t^*_n)}^2 + \frac{\mm}{\rho}\abs{z''_n(t^*_n)}^2 \le
    \\
    C_{15}\left(C_{14} + c_{12} (\norm{\phi}^2_{W^{2,2}_T} +
    \norm{g'}_{L^\infty_T} + \norm{\V
      f'}_{L^\infty(0,T;L^2(\Omega))})\right),
  \end{multline*}
  which is \eqref{eq:primeBoundsHelper} above and so the argument follows
  similar to the above \emph{Case I}.
\end{proof}
\begin{thm} \label{thm:strong}
  Under the assumptions of Lemma \ref{lem:timeRegularity}, the solution $(\V v,
  z)$ satisfies, furthermore:
  \begin{equation*}
    \V v \in L^\infty(0,T;W^{2,2}(\Omega)).
  \end{equation*}
\end{thm}
\begin{proof}
  With the regularity obtained in Lemma \ref{lem:timeRegularity},
  \eqref{eq:weak1} can be written as
  \begin{align} \label{eq:bilinear}
    \frac{\mu}{\rho} (\V D(\V v), \V D(\V\psi)) = (\V h - \V v \cdot \grad\V v,
    \V\psi), && \V\psi \in \mathcal{D}_0^\infty \quad(\text{or by density, }
    \V\psi \in \mathcal{D}^1),
  \end{align}
  where
  \begin{gather*}
    \V h = \V f - \frac{\partial\V v}{\partial t} - \V V \cdot \grad\V v - \V v
    \cdot \grad\V V + \frac{dz}{dt} \V e_1 \cdot \grad(\V v + \V V) + \grad w,
    \\
    w(\V x, t) = \frac 1\rho (\mm\frac{d^2z}{dt^2}-\ks z - g) \frac{\theta(\V
      x)}{\int_\Gamma n_1\theta\,dS}, 
  \end{gather*}
  for some $\theta \in C_0^\infty(\Omega \cup \Gamma)$ such that $\int_\Gamma
  n_1\theta\,dS \ne 0$. It follows from Lemma \ref{lem:timeRegularity},
  H\"older inequality and various Sobolev embedding theorems that $\V h \in
  L^\infty(0,T; L^2(\Omega))$.

  Also by Remark \ref{rmk:DvGradv}, the bilinear form $(\V D(\V v), \V
  D(\V\psi))$, on the left hand side of \eqref{eq:bilinear}, is elliptic so if
  one can show that $\V v \cdot \grad\V v \in L^\infty(0,T; L^2(\Omega))$, the
  claim follows from standard elliptic regularity results; And this can be
  shown by a bootstrap argument similar to the one in \cite[Theorem
    3.3.8]{MR1846644}.
\end{proof}

\section{Acknowledgments}
Giusy Mazzone gratefully acknowledges the support of the Natural Sciences and Engineering Research Council of Canada (NSERC) through the NSERC Discovery Grant ``Partially dissipative systems with applications to fluid-solid interaction problems''.

\bibliographystyle{plain}
\bibliography{infinitePipeOscillator}

\begin{thebibliography}{10}

\bibitem{MR2196495}
H.~Beir\~{a}o~da Veiga.
\newblock Time periodic solutions of the {N}avier-{S}tokes equations in
  unbounded cylindrical domains---{L}eray's problem for periodic flows.
\newblock {\em Arch. Ration. Mech. Anal.}, 178(3):301--325, 2005.

\bibitem{bonheureGaldi}
D.~Bonheure and G.~P. Galdi.
\newblock Global weak solutions to a time-periodic body-liquid interaction
  problem, 2023.

\bibitem{MR2761514}
T.~A. Burton.
\newblock {\em Stability and periodic solutions of ordinary and functional
  differential equations}.
\newblock Dover Publications, Inc., Mineola, NY, 2005.
\newblock Corrected version of the 1985 original.

\bibitem{bloodvalve2}
A.~Cristoforetti, M.~Mas\'e, R.~Bonmassari, M.~Dallago, G.~Nollo, and
  F.~Ravelli.
\newblock A patient-specific mass-spring model for biomechanical simulation of
  aortic root tissue during transcatheter aortic valve implantation.
\newblock {\em Physics in Medicine \& Biology}, 64(8):085014, apr 2019.

\bibitem{MR2808162}
G.~P. Galdi.
\newblock {\em An introduction to the mathematical theory of the
  {N}avier-{S}tokes equations}.
\newblock Springer Monographs in Mathematics. Springer, New York, second
  edition, 2011.
\newblock Steady-state problems.

\bibitem{MR4559728}
G.~P. Galdi.
\newblock Small forced oscillation of a rigid body in a viscous liquid.
\newblock In {\em Recent advances in mechanics and fluid-structure interaction
  with applications---the {B}ong {J}ae {C}hung memorial volume}, Adv. Math.
  Fluid Mech., pages 57--68. Birkh\"{a}user/Springer, Cham, 2022.

\bibitem{MR3329019}
G.~P. Galdi, M.~Mohebbi, R.~Zakerzadeh, and P.~Zunino.
\newblock Hyperbolic-parabolic coupling and the occurrence of resonance in
  partially dissipative systems.
\newblock In {\em Fluid-structure interaction and biomedical applications},
  Adv. Math. Fluid Mech., pages 197--256. Birkh\"auser/Springer, Basel, 2014.

\bibitem{bloodvalve1}
P.~E. Hammer, M.~S. Sacks, P.~J. del Nido, and R.~D. Howe.
\newblock Mass-spring model for simulation of heart valve tissue mechanical
  behavior.
\newblock {\em Annals of Biomedical Engineering}, 39(6):1668–--1679, 2011.

\bibitem{MR0120968}
V.~I. Judovi{\v{c}}.
\newblock Periodic motions of a viscous incompressible fluid.
\newblock {\em Soviet Math. Dokl.}, 1:168--172, 1960.

\bibitem{MR3092957}
M.~Mohebbi and J.~C. Oliveira.
\newblock Existence of time-periodic solutions for a magnetoelastic system in
  bounded domains.
\newblock {\em J. Elasticity}, 113(1):113--133, 2013.

\bibitem{wave1}
L.~Parrinello, P.~Dafnakis, E.~Pasta, G.~Bracco, P.~Naseradinmousavi,
  G.~Mattiazzo, and A.~P.~S. Bhalla.
\newblock {An adaptive and energy-maximizing control optimization of wave
  energy converters using an extremum-seeking approach}.
\newblock {\em Physics of Fluids}, 32(11):113307, 2020.

\bibitem{MR4419355}
Clara Patriarca.
\newblock Existence and uniqueness result for a fluid-structure-interaction
  evolution problem in an unbounded 2{D} channel.
\newblock {\em NoDEA Nonlinear Differential Equations Appl.}, 29(4):Paper No.
  39, 38, 2022.

\bibitem{MR1846644}
R.~Temam.
\newblock {\em Navier-{S}tokes equations}.
\newblock AMS Chelsea Publishing, Providence, RI, 2001.
\newblock Theory and numerical analysis, Reprint of the 1984 edition.

\bibitem{MR816732}
E.~Zeidler.
\newblock {\em Nonlinear functional analysis and its applications. {I}}.
\newblock Springer-Verlag, New York, 1986.
\newblock Fixed-point theorems, Translated from the German by Peter R. Wadsack.

\end{thebibliography}

\end{document}